%% file: ACC19_Cluster_v2.tex
\documentclass[letterpaper, 10 pt, conference]{ieeeconf}  % Comment this line out
\IEEEoverridecommandlockouts                              % This command is only
\usepackage[margin=0cm]{caption}
\usepackage{epsf}
\usepackage{epsfig}
\usepackage{times}
\usepackage{amssymb}
\usepackage{amsmath}
\usepackage{amsfonts}
\usepackage{latexsym}
\usepackage{url}
\usepackage{mathrsfs}
\usepackage{algorithm}
\usepackage{caption}
\usepackage{datetime}
\usepackage{bm}
\usepackage{algorithm,algorithmic}
\usepackage{subfigure}
\usepackage{framed} 
\usepackage[framed]{ntheorem}
\usepackage{cite}

\usepackage{lipsum}

\usepackage{amstext}  
\usepackage{array}      

\usepackage{mathabx}

\usepackage{mathtools}

\usepackage{color}
                        
\newframedtheorem{frm-definition}{Definition}

\newtheorem{theorem}{Theorem}
\newtheorem{definition}{Definition}
\newtheorem{lemma}{Lemma}

\newtheorem{remark}{Remark}
\newtheorem{assumption}{Assumption}

\newcommand{\hL}{\mathcal{L}}
\newcommand{\hN}{\mathcal{N}}

\newcommand{\cL}{{\cal{L}}}
\newcommand{\cN}{{\cal N}}

\newcommand{\cT}{{\cal T}}

\newcommand{\cE}{{\cal E}}

\newcommand{\cK}{{\cal K}}

\newcommand{\cU}{{\cal U}}

\newcommand{\cY}{{\cal Y}}

\begin{document}
\title{Hierarchical Distributed Voltage Regulation in Networked Autonomous Grids}
\author{Xinyang Zhou*, Zhiyuan Liu*, Wenbo Wang, Changhong Zhao, Fei Ding, Lijun Chen
\thanks{*The first two authors contributed equally.}
\thanks{X. Zhou, C. Zhao, and F. Ding are with the National Renewable Energy Laboratory, Golden, CO, 80401, USA (Emails: \{xinyang.zhou, changhong.zhao, fei.ding\}@nrel.gov).}
\thanks{Z. Liu and L. Chen are with the College of Engineering and Applied Science, University of Colorado, Boulder, CO 80309, USA (Emails: \{zhiyuan.liu, lijun.chen\}@colorado.edu).}
\thanks{W. Wang is with the Department of Electrical and Computer Engineering, New York University, Brooklyn, NY 11201, USA (Email: ww787@nyu.edu)}
%\thanks{This work was supported by U.S. Department of Energy through the ENERGISE program and the Advanced Grid Modeling program.}
}

\maketitle

\begin{abstract}
We propose a novel algorithm to solve optimal power flow (OPF) that aims at dispatching controllable distributed energy resources (DERs) for voltage regulation at minimum cost. The proposed algorithm features unprecedented scalability to large distribution networks by utilizing an information structure based on networked autonomous grids (AGs). Specifically, each AG is a subtree of a large distribution network that has a tree topology. The topology and line parameters of each AG are known only to a regional coordinator (RC) that is responsible for communicating with and dispatching the DERs within this AG. The reduced network, where each AG is treated as a node, is managed by a central coordinator (CC), which knows the topology and line parameters of the reduced network only and communicates with all the RCs. We jointly explore this information structure and the structure of the linearized distribution power flow (LinDistFlow) model to derive a hierarchical, distributed implementation of the primal-dual gradient algorithm that solves the OPF. The proposed implementation significantly reduces the computation burden compared to the centrally coordinated implementation of the primal-dual algorithm.  Numerical results on a 4,521-node test feeder show that the proposed hierarchical distributed algorithm can achieve an improvement of more than tenfold in the speed of convergence compared to the centrally coordinated primal-dual algorithm. 
\end{abstract}
% {\color{red}Xinyang: Waiting for acceptance on arXiv before submission.}
\input{introduction_v1.tex}

\section{System Modeling and Problem Formulation}\label{sec:model}
\subsection{Power Flow in Distribution Network}
Consider a radial power distribution network denoted by $\cT=\{\cN \cup \{0\},\cE\}$ with $N+1$ nodes collected in the set $\cN \cup \{0\}$, where $\cN:=\{1, ..., N\}$ and node $0$ is the slack bus, and distribution lines collected in the set $\cE$. For each node $i\in \hN$, denote by $\cE_i \subseteq \cE$ the set of lines on the unique path from node $0$ to node $i$, and let $p_i$ and $q_i$ denote the real and reactive power injected, where negative power injection means power consumption and positive power injection means power generation. Let $v_i$ be the magnitude of the complex voltage (phasor) at node $i$. For each line $(i, j)\in \cE$, denote by $r_{ij}$ and $x_{ij}$ its resistance and reactance, and  $P_{ij}$ and $Q_{ij}$ the real and reactive power from node $i$ to node $j$. Let $\ell_{ij}$ denote the squared magnitude of the complex branch current (phasor) from node $i$ to $j$. %We summarize some of the notations used in this paper in Section~\ref{sec:not}.

We adopt the following DistFlow model \cite{baran1989optimala, baran1989optimalb} for the radial distribution network:
\begin{subequations}\label{eq:bfm}
	\begin{eqnarray}
		\hspace{-5mm} P_{ij} \hspace{-2mm} &=&\hspace{-2mm} - p_j +\hspace{-2mm}\sum_{k: (j,k)\in \cE} \hspace{-2mm} P_{jk}+  r_{ij}  \ell_{ij}  \label{p_balance}, \\
		\hspace{-5mm}Q_{ij} \hspace{-2mm}&=&\hspace{-2mm}  -q_j + \hspace{-2mm}\sum_{k: (j,k)\in \cE} \hspace{-2mm} Q_{jk} + x_{ij} \ell_{ij} \label{q_balance},\\
		\hspace{-5mm}v_j^2 \hspace{-2mm}&=&\hspace{-2mm}  v_i^2 - 2 \big(r_{ij} P_{ij} + x_{ij} Q_{ij} \big) + \big(r_{ij}^2+x_{ij}^2\big) \ell_{ij} \label{v_drop},\\[1pt]
		\hspace{-5mm}\ell_{ij}v_i^2 \hspace{-2mm}&=&\hspace{-2mm}   P_{ij}^2 + Q_{ij}^2  \label{currents}.
	\end{eqnarray}
\end{subequations}

Following \cite{baran1989network, zhou2018reverse,liu2018signal}, we assume that the active and reactive power loss $r_{ij} \ell_{ij}$ and $ x_{ij} \ell_{ij}$, as well as $r^2_{ij} \ell_{ij}$ and $x^2_{ij} \ell_{ij}$, are negligible and can thus be ignored.   
%the terms involving $\ell_{ij}$ are zero for all $(i,j) \in \cE$ in \eqref{eq:bfm}. This approximation neglects the higher order real and reactive power loss terms. 
Indeed, the losses are much smaller than power flows $P_{ij}$ and $Q_{ij}$, typically on the order of $1\%$. We further assume that $v_i \approx 1,\ \forall i$ so that we can approximate $v_j^2 - v_i^2 \approx 2 (v_j - v_i)$ in Eq.~\eqref{v_drop}.\footnote{This assumption is not essential because we could work with $v_i^2$ instead.} This approximation introduces a small relative error of at most $0.25\%$ under the practically maximum $5\%$ deviation in voltage magnitude.

With these approximations, Eqs.~\eqref{eq:bfm} is simplified to the following linear model:
\begin{eqnarray}
\bm{v}&=&R\bm{p}+X\bm{q}+\tilde{\bm{v}},\label{eq:lindistflow}
\end{eqnarray}
where bold symbols $\bm{v}=[v_1,\ldots,v_N]^{\top}$, $\bm{p}=[p_1,\ldots,p_N]^{\top}$, $\bm{q}=[q_1,\ldots,q_N]^{\top}\in\mathbb{R}^N$ represent vectors, $\tilde{\bm{v}}$ is a constant vector depending on initial conditions, and the sensitivity matrices $R$ and $X$, respectively, consist of elements:
\begin{eqnarray}
	R_{ij}:= \!\!\! \sum_{(h,k)\in \cE_i \cap \cE_j}\!\!\!\! r_{hk}, \ \ \ \ X_{ij}:=\!\!\!\! \sum_{(h,k)\in \cE_i \cap \cE_j}\!\!\!\! x_{hk}  \label{X_def}.
\end{eqnarray}
Here, the voltage-to-power-injection sensitivity factors $R_{ij}$ ($X_{ij}$) represent \textit{the resistance (reactance) of the common path of node $\bm{i}$ and $\bm{j}$ leading back to node 0}. Keep in mind that this result serves as the basis for designing the hierarchical distributed algorithm to be introduced later. Fig.~\ref{fig:RX} illustrates $\cE_i \cap \cE_j$ for two arbitrary nodes $i$ and $j$ in a radial network and their corresponding $R_{ij}$ and $X_{ij}$.
\begin{figure}[htbp]
	\centering
	\includegraphics[width=.3\textwidth]{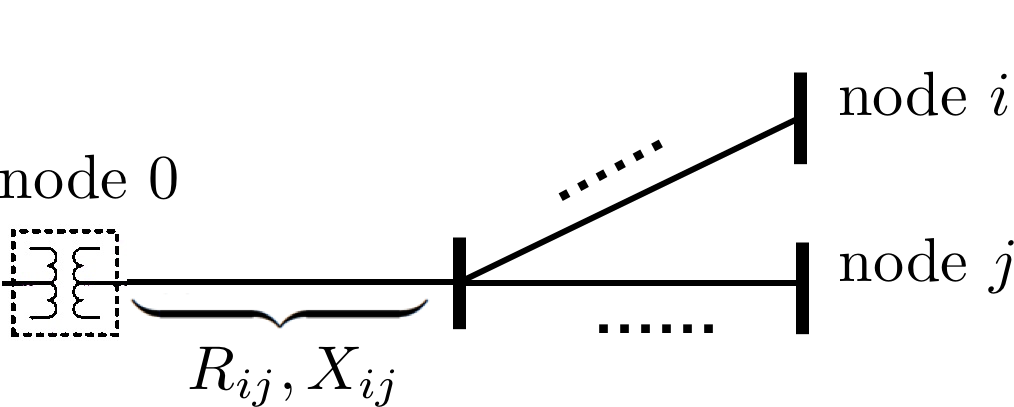}
	\caption{$\cE_i \cap \cE_j$ for  two arbitrary nodes $i, j$ in the network and the corresponding mutual voltage-to-power-injection sensitivity factors $R_{ij}, X_{ij}$. }
	\label{fig:RX}
\end{figure}

\subsection{OPF Problem}
Assume at node $i\in\cN$ there is a dispatchable DER (or aggregation of DERs) whose real and reactive power injections are confined as $(p_i,q_i)\in\cY_i$, where $\cY_i$ is a convex and compact set. Let $P_0$ denote active power injected  into the feeder at node 0, which can be approximated by the total active power loads as:
\begin{eqnarray}
 P_0=-P_I-\sum_{i\in\cN}p_i,\label{eq:P0}
\end{eqnarray}
where $P_I$ denotes the uncontrollable total inelastic loads. Note that $P_0$ is negative if the total active power consumption is larger than the total active power generation.
Use $\bm{v}(\bm{p},\bm{q})$ and $P_0(\bm{p})$ to represent \eqref{eq:lindistflow} and \eqref{eq:P0}, respectively, and consider the following OPF problem:
\begin{subequations}\label{eq:opt}
\begin{eqnarray}
&\underset{\bm{p},\bm{q}}{\min} & \sum_{i\in\cN}C_i(p_i,q_i)+C_0(P_0(\bm{p})),\\%+C_0(p,q)\\    
& \text{s.t.}&
 %\bm{v}=R\bm{p}+X\bm{q}+\tilde{\bm{v}},\label{eq:vlinear}\\
\underline{\bm{v}}\leq \bm{v}(\bm{p},\bm{q}) \leq \overline{\bm{v}}, \label{eq:voltreg}\\
%&     & P_0=\sum_{i\in\cN}p_i,\label{eq:P0}\\
&     & (p_i,q_i)\in\cY_i,\forall i\in\cN,\label{eq:X}
\end{eqnarray}
\end{subequations}
where the objective $C_i(p_i,q_i)$ is the cost function for node $i$, and the coupling term $C_0(P_0(\bm{p}))$ represents the cost associated with the total network load. For example, $C_0(P_0(\bm{p}))=\alpha (P_0(\bm{p})-\hat{P}_0)^2$ penalizes $P_0(\bm{p})$'s deviation from a dispatching signal $\hat{P}_0$ with a weight $\alpha>0$. We make the following assumption for these cost functions.
\begin{assumption}\label{ass:costfun}
$C_i(p_i,q_i),\: \forall i\in\cN$ are continuously differentiable and strongly convex in $(p_i,q_i)$, with bounded first-order derivative in $\cY_i$. Meanwhile, $C_0(P_0)$ is continuously differentiable and convex with bounded first-order derivative.
\end{assumption}

\subsection{Primal-Dual Gradient Algorithm}
Associate dual variables $\underline{\bm{\mu}}$ and $\overline{\bm{\mu}}$ with the left-hand-side and the right-hand-side of \eqref{eq:voltreg}, respectively, so that the Lagrangian of \eqref{eq:opt} is written as:
\begin{eqnarray}
    \cL(\bm{p},\bm{q};\overline{\bm{\mu}},\underline{\bm{\mu}})=\sum_{i\in\cN}C_i(p_i,q_i)+C_0(P_0(\bm{p}))\nonumber\\
    +\underline{\bm{\mu}}^{\top}(\underline{\bm{v}}-\bm{v}(\bm{p},\bm{q}))+\overline{\bm{\mu}}^{\top}(\bm{v}(\bm{p},\bm{q})-\overline{\bm{v}}),\label{eq:lang}
\end{eqnarray}
with 
%constraints \eqref{eq:vlinear} and \eqref{eq:P0} represented by $\bm{v}(\bm{p},\bm{q})$ and $P_0(\bm{p})$, and 
\eqref{eq:X} treated as the domain of $(\bm{p},\bm{q})$.

To design an algorithm with provable convergence, we introduce the following regularized Lagrangian with parameter $\phi>0$:
\begin{eqnarray}
    \cL_{\phi}(\bm{p},\bm{q};\overline{\bm{\mu}},\underline{\bm{\mu}})=\sum_{i\in\cN}C_i(p_i,q_i)+C_0(P_0(\bm{p}))\nonumber\\[-5pt]
    +\underline{\bm{\mu}}^{\top}(\underline{\bm{v}}-\bm{v}(\bm{p},\bm{q}))+\overline{\bm{\mu}}^{\top}(\bm{v}(\bm{p},\bm{q})-\overline{\bm{v}})-\frac{\phi}{2}\|\bm{\mu}\|^2_2,\label{eq:langr}
\end{eqnarray}
where $\bm{\mu} :=[\underline{\bm{\mu}}^{\top}, \overline{\bm{\mu}}^{\top}]^{\top}$.
\begin{theorem}\label{the:unique}
There exists one unique saddle point $(\bm{p}^*,\bm{q}^*;\underline{\bm{\mu}}^*,\overline{\bm{\mu}}^*)$ of $\cL_{\phi}$.
\end{theorem}
\begin{proof}
The result follows from the fact that $\cL_{\phi}(\bm{p},\bm{q};\underline{\bm{\mu}},\overline{\bm{\mu}})$ is strongly convex in $\bm{p},\bm{q}$, and strongly concave in $\underline{\bm{\mu}},\overline{\bm{\mu}}$. 
\end{proof}

\begin{theorem}\label{the:regularization}
The difference between the saddle points of \eqref{eq:lang} and \eqref{eq:langr} is bounded, and it is proportional to $\phi$. 
\end{theorem}

We refer the proof of Theorem~\ref{the:regularization} to \cite{koshal2011multiuser,zhou2017incentive}.

Then, the iterative primal-dual gradient algorithm to find the saddle point of \eqref{eq:langr} is cast as follows:
\begin{subequations}\label{eq:primaldual}
\begin{eqnarray}
    \hspace{-2mm}\bm{p}(t+1)\hspace{-3mm}&=&\hspace{-3mm}\Big[\bm{p}(t)-\epsilon\Big(\!\nabla_{\bm{p}} C(\bm{p}(t),\!\bm{q}(t))\!-\!C'_0(P_0(\bm{p}(t)))\!\cdot\! \bm{1}_N\nonumber\\[-2pt]
    \hspace{-2mm}&&\hspace{22mm}+R^{\top}(\overline{\bm{\mu}}(t)-\underline{\bm{\mu}}(t))\Big)\Big]_{\bm{\cY}},\label{eq:primalP}\\[-2pt]
    \hspace{-2mm}\bm{q}(t+1)\hspace{-3mm}&=&\hspace{-3mm}\Big[\bm{q}(t)-\epsilon\Big(\!\nabla_{\bm{q}} C(\bm{p}(t),\bm{q}(t)) \nonumber\\[-2pt]
    \hspace{-2mm}&&\hspace{22mm}+X^{\top}(\overline{\bm{\mu}}(t)-\underline{\bm{\mu}}(t))\Big)\Big]_{\bm{\cY}},\label{eq:primalQ}\\[-2pt]
    \hspace{-2mm}\underline{\bm{\mu}}(t+1)\hspace{-3mm}&=&\hspace{-3mm}\big[\underline{\bm{\mu}}(t)+\epsilon (\underline{\bm{v}}-\bm{v}(t)-\phi   \underline{\bm{\mu}}(t))\big]_+,\label{eq:dual1}\\
    \hspace{-2mm}\overline{\bm{\mu}}(t+1)\hspace{-3mm}&=&\hspace{-3mm}\big[\overline{\bm{\mu}}(t)+\epsilon (\bm{v}(t)-\overline{\bm{v}}-  \phi\overline{\bm{\mu}}(t))\big]_+,\label{eq:dual2}
\end{eqnarray}
\end{subequations}
where $\epsilon>0$ is some constant stepsize to be determined; $\bm{1}_N=[1,\ldots,1]^{\top}\in\mathbb{R}^N$, $[\  ]_{\bm{\cY}}$ is the projection operator onto the feasible set $\bm{\cY}:=\bigtimes_{i\in\cN}\cY_i$; and $[ \ ]_{+}$ is the projection operator onto the positive orthant.

\subsection{Convergence}
We use $\bm{y}:=[\bm{p}^{\top},\bm{q}^{\top}]^{\top}$ to stack all the primal variables and rewrite \eqref{eq:primaldual} equivalently as follows:
\begin{eqnarray}
\begin{bmatrix}\bm{y}(t+1)\\ \bm{\mu}(t+1)\end{bmatrix} =\left[\begin{bmatrix}\bm{y}(t)\\ \bm{\mu}(t)\end{bmatrix}-\epsilon\begin{bmatrix}\nabla_{\bm{y}} \hL_{\phi}(\bm{y}(t),\bm{\mu}(t)) \\ -\nabla_{\bm{\mu}}\hL_{\phi}(\bm{y}(t),\bm{\mu}(t))
\end{bmatrix} \right]_{\bm{\cY}\times \bm{\cU}}\hspace{-4mm},\label{eq:mapk}
\end{eqnarray}
where $\bm{\cU}$ is the feasible positive orthant for the dual variables.
We further let $\bm{z}:=[\bm{y}^{\top},\bm{\mu}^{\top}]^{\top}$ stack all variables, and $T(\bm{z})$ denote the operator $\begin{bmatrix}\nabla_{\bm{y}} \hL_{\phi}(\bm{y},\bm{\mu}) \\ -\nabla_{\bm{\mu}}\hL_{\phi}(\bm{y},\bm{\mu})\end{bmatrix}$.

\begin{lemma}\label{lem:mono}
$T(\bm{z})$ is a strongly monotone operator.
\end{lemma}
\begin{proof}
	See Appendix.
\end{proof}

By Lemma~\ref{lem:mono}, there exists some constant $M>0$ such that for any $\bm{z},\bm{z'}\in\bm{\cY}\times \bm{\cU}$, one has:
\begin{eqnarray}
(T(\bm{z})-T(\bm{z'}))^{\top}(\bm{z}-\bm{z'})\geq M \|\bm{z}-\bm{z'}\|_2^2.\label{eq:strongmono}
\end{eqnarray}
Moreover, based on Assumption~\ref{ass:costfun}, the operator $T(\bm{z})$ is also Lipschitz continuous, i.e., there exists some constant $L>0$ such that for any feasible $\bm{z}$ and $\bm{z'}$, we have:
\begin{eqnarray}
\|T(\bm{z})-T(\bm{z'})\|_2^2\leq L^2 \|\bm{z}-\bm{z'}\|_2^2.\label{eq:lipschitz}
\end{eqnarray}
We present the next theorem that ensures the convergence of \eqref{eq:mapk} with a small enough stepsize. 
\begin{theorem}\label{the:converge}
If the stepsize $\epsilon$ satisfies $0<\epsilon\leq \overline{\epsilon}<2M/L^2$ for some $\overline{\epsilon}$, \eqref{eq:mapk} converges to the unique saddle point of \eqref{eq:langr} exponentially fast.
\end{theorem}
\begin{proof}
	See Appendix.
\end{proof}

\subsection{Motivation for Hierarchical Design}
Note that in \eqref{eq:primaldual} the update of any $p_i$ (resp. $q_i$) involves the knowledge of $\sum_{j\in\cN}R_{ij}(\overline{\mu}_j-\underline{\mu}_j)$ (resp. $\sum_{j\in\cN}X_{ij}(\overline{\mu}_j-\underline{\mu}_j)$). Therefore, at each iteration a coordinator cognizant of the entire network  sensitivity matrices $R$ and $X$ is required to first collect updated dual variables from all the nodes, calculate $R^{\top}(\overline{\bm{\mu}}-\underline{\bm{\mu}})$ and $X^{\top}(\overline{\bm{\mu}}-\underline{\bm{\mu}})$, and then send the results back to the corresponding nodes. 
This becomes computationally more challenging in a larger network containing thousands of or even more controllable endpoints, not to mention the recalculation of the huge $R,X$ matrices in case changes in network topology or regulator taps occur.

This motivates us to design a hierarchical control structure where the large network is partitioned into smaller AGs, each of which is managed locally by its own RC, and there is a CC that manages only a reduced network where each AG is treated as one node. As will be shown in the next section, the hierarchical control design not only distributes the computational burden but also reduces a huge amount of the computation by utilizing the network structure.

\begin{algorithm*}[h]
	\caption{Hierarchical Distributed Voltage Regulation} 
	\begin{algorithmic}\label{alg:distalg}
		
		\REPEAT
	
		\STATE[1] node $i\in\cN$ updates $(p_i(t+1),q_i(t+1))$ by
		\begin{subequations}
			\begin{eqnarray}
	        	p_i(t+1)&=&[p_i(t)-\epsilon(\partial_{p_i} C_i(p_i(t),q_i(t)) - C'_0(P_0(\bm{p}(t))) +\alpha_{i}(t))]_{\cY_i},\\ q_i(t+1)&=&[q_i(t)-\epsilon(\partial_{q_i} C_i(p_i(t),q_i(t)) +\beta_{i}(t))]_{\cY_i},
	    	\end{eqnarray}
		\end{subequations}
	
		\hspace{5mm} and $(\underline{\mu}_i(t+1),\overline{\mu}_i(t+1))$ based on local voltage by
		\begin{eqnarray}
	 		 \underline{\mu}_i(t+1)=[\underline{\mu}_i(t)+\epsilon (\underline{v}_i-v_i(t)-\phi \underline{\mu}_i(t))]_+,\ \  \overline{\mu}_i(t+1)=[\overline{\mu}_i(t)+\epsilon (v_i(t)-\overline{v}_i-\phi \overline{\mu}_i(t))]_+.
		\end{eqnarray}
		
		\STATE[2] 
	%	\colorbox[RGB]{239,240,241}{\hspace{-2mm}
		RC $k\in\cK$ calculates and sends $\underset{i\in\cN_k}{\sum}\big(\overline{\mu}_i(t+1)-\underline{\mu}_i(t+1)\big)$ to CC; unclustered node $i\in\cN_0$ sends $(\overline{\mu}_i(t+1)-\underline{\mu}_i(t+1))$ to CC.

		\STATE[3] CC 
		%collects $\sum_{j\in\cN_k}(\overline{\mu}_j(t)-\underline{\mu}_j(t)),\forall k\in\cK$, together with $(\overline{\mu}_j-\underline{\mu}_j),\forall j\in\cN_0$ 
		computes within the reduced network
		\begin{subequations}
			\begin{eqnarray}
			\alpha_k^{\text{out}}(t+1)&=&\hspace{-3mm}\sum_{h\in\cK,h\neq k}R_{n^0_h n^0_k}\sum_{j\in\cN_h}(\overline{\mu}_j(t+1)-\underline{\mu}_j(t+1))+\sum_{j\in\cN_0}R_{jn_k^0}(\overline{\mu}_j(t+1)-\underline{\mu}_j(t+1)),\ \forall k\in\cK,\\
			\beta_k^{\text{out}}(t+1)&=&\hspace{-3mm}\sum_{h\in\cK,h\neq k}X_{n^0_h n^0_k}\sum_{j\in\cN_h}(\overline{\mu}_j(t+1)-\underline{\mu}_j(t+1))+\sum_{j\in\cN_0}X_{jn_k^0}(\overline{\mu}_j(t+1)-\underline{\mu}_j(t+1)),\ \forall k\in\cK,
			\end{eqnarray}
		\end{subequations}
		\hspace{5mm} and \vspace{-7mm}
		\begin{subequations}
		\begin{eqnarray}
		\alpha_i(t+1)&=&\sum_{k\in\cK}R_{i n^0_k}\sum_{j\in\cN_k}(\overline{\mu}_j(t+1)-\underline{\mu}_j(t+1))+\sum_{j\in\cN_0}R_{i j} (\overline{\mu}_j(t+1)-\underline{\mu}_j(t+1)),\ \forall i\in\cN_0,\\
		\beta_i(t+1)&=& \sum_{k\in\cK}X_{i n^0_k}\sum_{j\in\cN_k}(\overline{\mu}_j(t+1)-\underline{\mu}_j(t+1))+\sum_{j\in\cN_0}X_{i j} (\overline{\mu}_j(t+1)-\underline{\mu}_j(t+1)),\ \forall i\in\cN_0,
		\end{eqnarray}
	\end{subequations}
		\hspace{5mm} and sends $(\alpha_k^{\text{out}}(t+1),\beta_k^{\text{out}}(t+1))$ to RC $k\in\cK$,  and $(\alpha_i(t+1),\beta_i(t+1))$ to unclustered node $i\in\cN_0$.

			\STATE[4] RC $k\in\cK$ calculates within AG $k$:
		\begin{subequations}
			\begin{eqnarray}
			&\alpha^{\text{in}}_{k,i}(t+1)=\sum_{j\in\cN_k}R_{ij}(\overline{\mu}_j(t+1)-\underline{\mu}_j(t+1)),\ \text{and} &\alpha_{i}(t+1)=\alpha^{\text{in}}_{k,i}(t+1)+\alpha^{\text{out}}_k(t+1),\ \forall i\in\cN_k,
			\\
			&\beta^{\text{in}}_{k,i}(t+1)=\sum_{j\in\cN_k}X_{ij}(\overline{\mu}_j(t+1)-\underline{\mu}_j(t+1)),\ \text{and} &\beta_{i}(t+1)=\beta^{\text{in}}_{k,i}(t+1)+\beta^{\text{out}}_k(t+1),\ \forall i\in\cN_k,
			\end{eqnarray}
		\end{subequations}
		\hspace{5mm} and sends $(\alpha_i(t+1),\beta_{i}(t+1))$ to node $i\in\cN_k$.

		\STATE[5] $\bm{v}(t+1)$ and $P_0(t+1)$ are updated by the physical system:
		\begin{eqnarray}
		\bm{v}(t+1)=R\bm{p}(t+1)+X\bm{q}(t+1)+\tilde{\bm{v}},\ \ \ \ \  P_0(t+1)=P_I+\sum_{i\in\cN}p_i(t+1).
		\end{eqnarray}
		
		\STATE[6] CC computes/measures $P_0(t+1)$ at the substation and broadcasts $C'_0 (P_0(\bm{p}(t+1)) )$.
		
		\UNTIL stopping criterion is met (e.g., $|P_0(\bm{p}(t+1))-P_0(\bm{p}(t))|<\sigma$ for some small $\sigma>0$)
	\end{algorithmic}
\end{algorithm*}

\section{Hierarchical Distributed Algorithm}\label{sec:hier}

\subsection{Distribution Feeder as Networked Autonomous Grids}

\begin{definition}
A subtree of a tree $\cT$ is a tree consisting of a node in $\cT$, \textbf{all of its descendants} in $\cT$, and their connecting lines.
\end{definition}

We group all the nodes of the distribution network $\cT$ into (1) $K$ subtrees indexed by $\cT_k=\{\cN_k,\cE_k\},\ k\in\cK=\{1,\ldots,K\}$, and (2) a set $\cN_0$ collecting all the other ``unclustered" nodes in $\cN$. Here, $\cN_k$ of size $N_k$ is the set of nodes in subtree $\cT_k$ and $\cE_k$ contains their connecting lines. Thus, we have $\cup_{k\in\cK}\cN_k\cup\cN_0=\cN$ and $\cN_j \cap \cN_k=\emptyset,\forall j\neq k$.
Assume each subtree $\cT_k$ is an AG managed by an RC cognizant of the topology of $\cT_k$ and communicating with all the controllable nodes within $\cT_k$.

Denote the root node of subtree $\cT_k$ by $n_k^0$, and consider a reduced network $\cT^r=\{\cN^r\cup\{0\},\cE^r\}$ where $\cN^r:=\cup_{k\in\cK}\{n_k^0\}\cup \cN_0$ consists of the root nodes of subtrees and all the unclustered nodes, and $\cE^r$ consists of their connecting lines. 
We assume a CC cognizant of the topology of the reduced network $\cT^r$ and communicating with all the RCs as well as the unclustered nodes. 

Since each indexed subtree is considered as an AG, in this paper, we use the terms ``subtree" and ``AG" interchangeably.

  \begin{figure}[h]
 	\centering
 	\includegraphics[trim = 0mm 0mm 0mm 0mm, clip, scale=0.4]{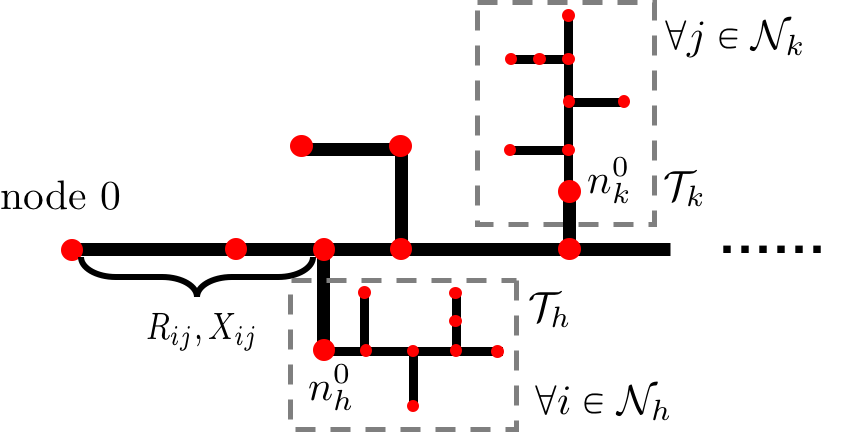}
 	\caption{The unclustered nodes and the root nodes of subtrees together with their connecting lines constitute the reduced network. Two subtrees $\cT_h$ and $\cT_k$ share the same $R_{ij}$ and $X_{ij}$ for any of their respective nodes $i$ and $j$.}
 	\label{fig:RijXij}
 \end{figure}

\subsection{Hierarchical Distributed Algorithm}

In this part, we design a hierarchical distributed algorithm for the networked AG structure introduced previously, by exploring the structure of the network matrices $R$ and $X$.
For simplicity, we elaborate the algorithm design for real power injections $\bm{p}$ only, and that of $\bm{q}$ follows similarly.

Rewrite \eqref{eq:primalP} as:
\begin{eqnarray}
    &&\hspace{-8mm}p_i(t+1)=\Big[p_i(t)-\epsilon\Big(\partial_{p_i} C_i(p_i(t),q_i(t))-C'_0(P_0(\bm{p}(t)))\nonumber\\
    &&\hspace{10mm}+\sum_{j\in\cN}R_{ij}\big(\overline{\mu}_j(t)-\underline{\mu}_j(t)\big)\Big) \Big]_{\cY_i}, \forall i\in\cN.\label{eq:recastP}
\end{eqnarray}
Note that in \eqref{eq:recastP}, although $\partial_{p_i} C_i(p_i(t),q_i(t))$ is local information and the scalar $C'_0(P_0(\bm{p}(t)))$ can be easily broadcast, the last term $\sum_{j\in\cN}R_{ij}(\overline{\mu}_j(t)-\underline{\mu}_j(t))$ couples the whole network in principle.

To design a more scalable algorithm, we first introduce the following lemma.

\begin{lemma}\label{lem:commonpath}
Given any two subtrees $\cT_h$ and $\cT_k$ with their root nodes $n^0_h$ and $n^0_k$, we have $R_{ij}=R_{n^0_h n^0_k},X_{ij}=X_{n^0_h n^0_k}, \forall i\in\cN_h, \forall j\in\cN_k$. Similarly, given any unclustered node $i\in\cN_0$ and a subtree $\cT_k$ with its root node $n^0_k$, we have $R_{ij}=R_{i n^0_k},X_{ij}=X_{i n^0_k}, \forall j\in\cN_k$.
\end{lemma}
\begin{proof}
We have the two following facts:
	\begin{enumerate}
		\item By \eqref{X_def}, $R_{ij}$ (resp. $X_{ij}$) is the summed resistance (resp. reactance) of the common path of node $i$ and $j$ leading back to node 0.
		\item Any node in one subtree and any node in another subtree (or any node in one subtree and one unclustered node) share the same common path back to node 0.
	\end{enumerate}
The result is immediate from 1) and 2).
\end{proof}

\begin{remark}
	The results of Lemma~\ref{lem:commonpath} can be illustrated in Fig.~\ref{fig:RijXij}. In Fig.~\ref{fig:RijXij}, any node $i$ in $\cT_h$ and any node $j$ in $\cT_k$ share the same $R_{ij}=R_{n^0_h n^0_k}$ and $X_{ij}=X_{n^0_h n^0_k}$, which by definition~\eqref{X_def} are the summed resistance and reactance of their marked common path connected to node 0.
\end{remark}

Lemma~\ref{lem:commonpath} enables us to recalculate the coupling terms in a hierarchical distributed way.

\noindent\textbf{For node $i\in\cN_k$:}  We decompose $\sum_{j\in\cN}R_{ij}(\overline{\mu}_j-\underline{\mu}_j)$ as:
\begin{eqnarray}
&&\sum_{j\in\cN}R_{ij}(\overline{\mu}_j-\underline{\mu}_j)\nonumber\\
&=& \sum_{j\in\cN_k}R_{ij}(\overline{\mu}_j-\underline{\mu}_j)+\hspace{-2mm}\sum_{j\in\cN\backslash\cN_k}\hspace{-2mm}R_{ij}(\overline{\mu}_j-\underline{\mu}_j)\nonumber\\
&=& \sum_{j\in\cN_k}R_{ij}(\overline{\mu}_j-\underline{\mu}_j)+\hspace{-2mm}\sum_{h\in\cK,h\neq k}\hspace{-3mm}R_{n^0_h n^0_k}\sum_{j\in\cN_h}(\overline{\mu}_j-\underline{\mu}_j)\nonumber\\
&&\hspace{4mm}+\sum_{j\in\cN_0}R_{n_k^0j}(\overline{\mu}_j-\underline{\mu}_j),\label{eq:decompose}
\end{eqnarray}
where the first part of \eqref{eq:decompose} consists of information within AG $k$ (together with the line parameter from $n_k^0$ to bus 0, i.e., $R_{n_k^0n_k^0}$ and $X_{n_k^0n_k^0}$, which can be informed by CC), the second from all the other AGs, and the third from the unclustered nodes. For convenience, denote:
\begin{eqnarray}
\hspace{-5mm}&&\alpha_{k,i}^{\text{in}}=\sum_{j\in\cN_k}R_{ij}(\overline{\mu}_j-\underline{\mu}_j),\nonumber\\
\hspace{-5mm}&&\alpha_k^{\text{out}}=\hspace{-2mm}\sum_{h\in\cK,h\neq k}\hspace{-2mm}R_{n^0_h n^0_k}\sum_{j\in\cN_h}(\overline{\mu}_j-\underline{\mu}_j)+\sum_{j\in\cN_0}R_{n_k^0j}(\overline{\mu}_j-\underline{\mu}_j).\nonumber
\end{eqnarray} 
Note that for $i\in\cN_k$, $\alpha_{k,i}^{\text{in}}$ is accessible by RC $k$ cognizant of the topology of AG $k$, and $\alpha_k^{\text{out}}$ does not involve the network structure of any other AGs but only that of the reduced network known by CC. 

\noindent \textbf{For unclustered node $i\in\cN_0$:} Similarly, one has:
\begin{eqnarray}
&\hspace{-10mm}&\sum_{j\in\cN}R_{ij}(\overline{\mu}_j-\underline{\mu}_j)\nonumber\\
&\hspace{-10mm}=& \sum_{k\in\cK}R_{i n^0_k}\sum_{j\in\cN_k}(\overline{\mu}_j-\underline{\mu}_j)+\sum_{j\in\cN_0}\hspace{-1mm}R_{ij}(\overline{\mu}_j-\underline{\mu}_j),\label{eq:decompose2}
\end{eqnarray}
whose computation requires only the topology and line parameters of the reduced network known by CC. 

Meanwhile, computational loads  are also largely reduced because of the following reasons: (1) the term $\sum_{h\in\cK,h\neq k}R_{n^0_h n^0_k}\sum_{j\in\cN_h}(\overline{\mu}_j-\underline{\mu}_j)$ requires less computation than the original $\sum_{j\in\cN_h,h\neq k}R_{ij}(\overline{\mu}_j-\underline{\mu}_j)$; and, more importantly, (2) $\alpha_k^{\text{out}}$ is the same for all the nodes in $\cN_k$, which reduces a lot of the repetitive computation that is executed by \eqref{eq:primaldual}.
Equations \eqref{eq:decompose}--\eqref{eq:decompose2} motivate us to design the hierarchical distributed algorithm, Algorithm~\ref{alg:distalg}. As will be shown in Section~\ref{sec:numerical}, {the improvement in convergence speed is more than ten times for a 4,521-node feeder, enabling fast real-time grid optimization}.

In Algorithm~\ref{alg:distalg}, both the CC and RCs are in charge of only a portion of the whole system: the CC manages the reduced network and coordinates RCs and unclustered nodes without knowing any detailed information within any AGs, and the RCs each manage their own AGs without knowing detailed information of other AGs or the reduced network. Also see Fig.~\ref{fig:testfeeder} for an illustration of clustering and information flows in Algorithm~\ref{alg:distalg}.

%Besides, if $R$ and $X$ change due to either topology alternation or regulator tap changes which take place quite often, there is no need to re-calculate the whole huge system matrices, but can be regionally updated by CC and RCs.
 
In addition, because Algorithm~~\ref{alg:distalg} is mathematically equivalent to \eqref{eq:primaldual}, the results in Theorems~\ref{the:unique}--\ref{the:converge} also apply to  Algorithm~\ref{alg:distalg}.

% \end{subequations}
 \begin{figure}
     \centering
     \includegraphics[scale=0.3]{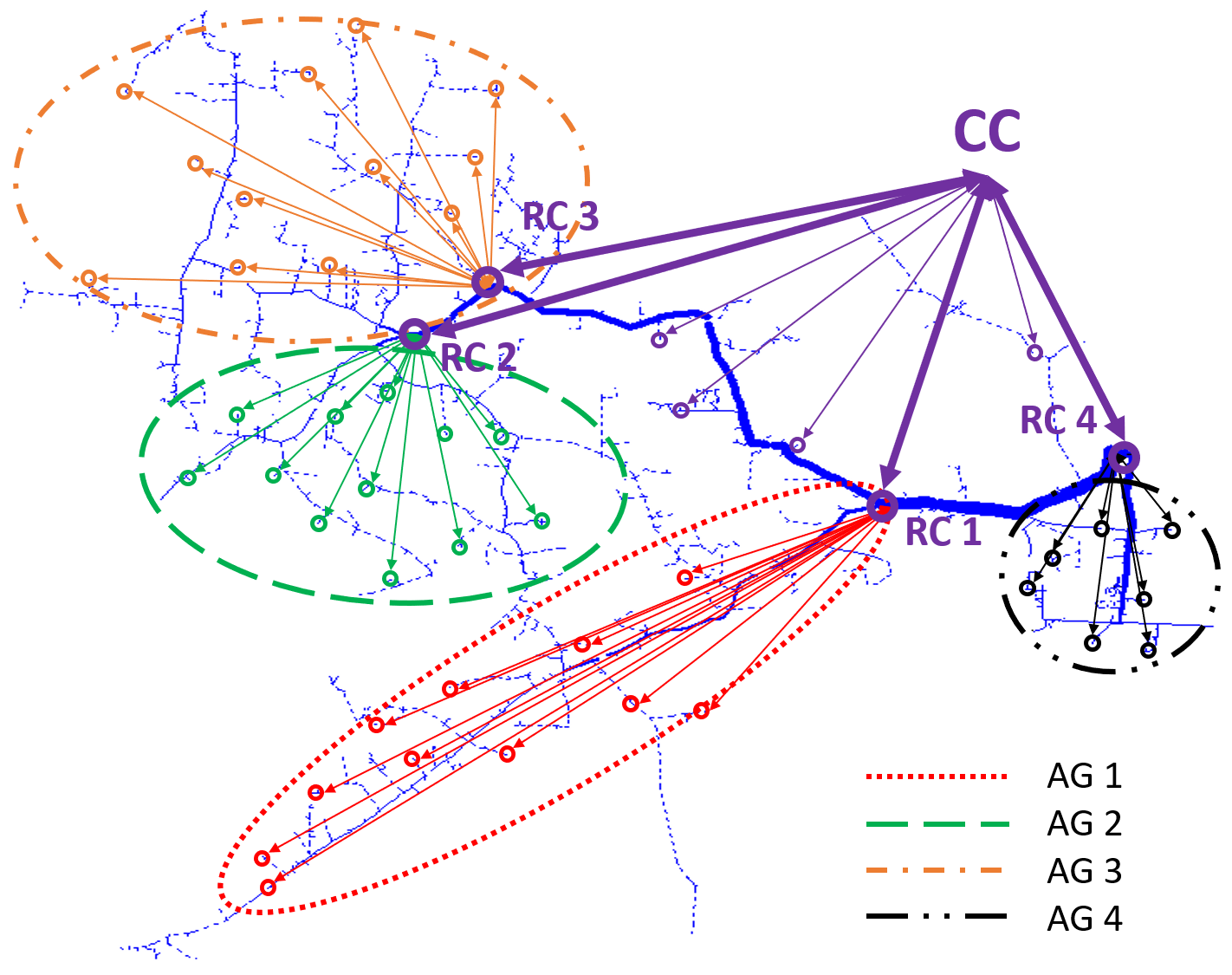}
     \caption{The 11,000-node test feeder constructed from an IEEE 8,500-node test feeder and a modified EPRI Test Circuits Ckt7. Four AGs are formed for our experiments.}\label{fig:testfeeder}
 \end{figure}
 
\section{Numerical Results}\label{sec:numerical}
In this section, we provide numerical results to show the effectiveness and the efficiency of our design. 
\subsection{Network Setup}
A three-phase, unbalanced, 11,000-node test feeder is constructed by connecting an IEEE 8,500-node test feeder and a modified EPRI Ckt7.  Fig.~\ref{fig:testfeeder} shows the single-line diagram of the feeder, where the line width is proportional to the nominal power flow on it. The primary side of the feeder is modeled in detail, whereas the loads on the secondary side (which in this system is aggregation of several loads) are lumped into corresponding distribution transformers, resulting in a 4,521-node network with 1,335 aggregated loads. 
%It is straightforward to extend the current framework to load control on the secondary transformer side so that they add up to the desired aggregated values. We choose not to conduct it here for simplicity; we refer to \cite{zhou2017discrete} for related control design.
We group all the nodes into unclustered nodes and four AGs marked in Fig.~\ref{fig:testfeeder}. AG 1 contains 357 nodes with controllable loads, AG 2 contains 222, AG 3 contains 310, and AG 4 contains 154. We fix the loads on all 292 unclustered nodes for simplicity.

The three-phase nonlinear power flow model is simulated in OpenDSS. With default control of the capacitors and regulators in OpenDSS \cite{dugan2010ieee}, we achieve the voltage profile shown in Fig.~\ref{fig:Voltage_controlled} with orange dots, where undervoltages are observed.
We next disable the control of all the capacitors and regulators to obtain the heavily undervoltage scenario shown with blue dots in Fig.~\ref{fig:Voltage_controlled}, and we implement our voltage regulation algorithm.

% According to the aforementioned control algorithm, the clustering strategy is applied first. Clustering: 12 single-phase sub-trees, each consisting of 300--500 (?) nodes.

% We consider voltage regulation within 0.95--1.05 p.u. for all 4528 nodes at primary distribution level.

\subsection{Hierarchical Distributed Voltage Regulation}
For each controllable node $i$, we consider minimizing the cost of its deviation from its original load level $p_i^{o}, q_i^o$, i.e., $C_i(p_i,q_i)=(p_i-p_i^{o})^2+(q_i-q_i^{o})^2$. Because we focus on voltage regulation here, we set $C_0(P_0)=\alpha (P_0^2-\hat{P}_0)^2$ with a small $\alpha=0.0005$ and $\hat{P}_0$ set to 80\% of the initial $P_0$. $\underline{v}_i$ and $\overline{v}_i$ are uniformly set to $0.95$~p.u. and $1.05$~p.u., respectively. We implement Algorithm~\ref{alg:distalg} with a constant stepsize $\epsilon=1\times 10^{-3}$.

 \begin{figure}
     \centering
     \includegraphics[scale=0.5]{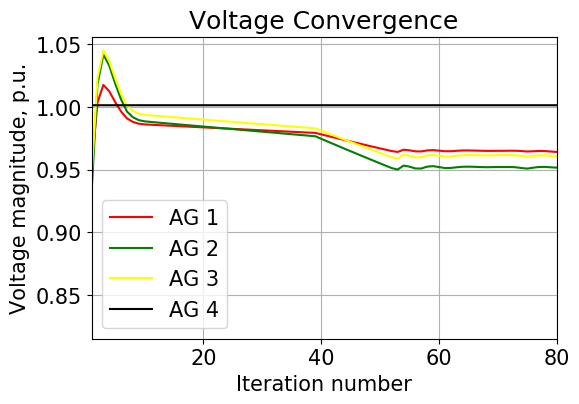}
     \caption{Voltage convergence at selected nodes from four AGs.}\label{fig:Voltage_converge}
 \end{figure}
 
\subsubsection{Convergence}
As shown in Fig.~\ref{fig:Voltage_converge}, convergence of the algorithm takes 50--60 iterations. In addition, thanks to the hierarchical design that largely reduces the computational burden, it takes about 25 seconds for the control of more than 1,000 nodes to converge on a laptop with an Intel Core i7-7600U CPU @ 2.80GHz 2.90GHz, 8.00GB RAM, running Python 3.6 on Windows 10 Enterprise Version. On the other hand, it takes the centrally coordinated primal-dual algorithm~\eqref{eq:primaldual}  (which is implemented by a centralized coordinator for all the nodes) almost four times longer to converge. By further considering the parallel computation of four clusters, the overall improvement in speed is more than  tenfold. %In this case, our design improves the convergence speed by more than 4 times without compromising accuracy of the OPF solution.

 \begin{figure}
	\centering
	\includegraphics[scale=0.5]{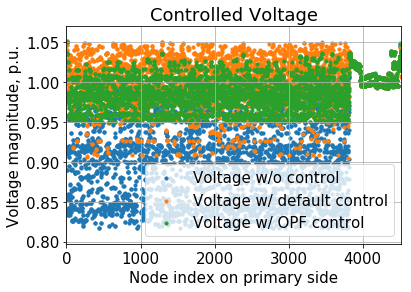}
	\caption{Voltages are controlled within $[0.95,1.05]$ p.u. by the proposed OPF algorithm.} \label{fig:Voltage_controlled}
\end{figure}

\subsubsection{Voltage Regulation} We plot the regulated voltages obtained by Algorithm~\ref{alg:distalg} with green dots in Fig.~\ref{fig:Voltage_controlled}. Note that the voltage magnitudes of all the nodes are strictly constrained within the $[0.95,1.05]$ p.u. bound. In contrast, the default control of regulators and capacitors cannot guarantee that all the voltages are within this bound.
%\textcolor{red}{Wenbo: We can see that the green dots are not only better than blue and yellow dots on average, but also distributed more uniformly.}

%  \begin{figure}
%	\centering
%	\includegraphics[scale=0.5]{powerloss_ratio2.png}
%	\caption{Power loss ratio reduction.}\label{fig:powerloss_converge}
%\end{figure}

%\subsubsection{Power Loss Minimization}
%The total power loss of the network is approximately proportional to the total network load $P_0$, so the $C_0(P_0)$ term in the OPF objective can effectively reduce power loss as well. Compared with the total power loss under the  uncontrolled case which is 6.5\% of the total network load, and the 7.3\% loss under default control of regulators and capacitors, our design lowers the loss to 3.8\%, as plotted in Fig.~\ref{fig:powerloss_converge}.

%  \begin{figure}
%      \centering

%      \caption{Lineloss ratio is minimized.}\label{fig:lineloss_ratio}
%  \end{figure}

\section{Conclusion}\label{sec:conclusion}
We proposed a hierarchical distributed implementation of the primal-dual gradient algorithm to solve an OPF problem. The objective of the OPF is to minimize the total cost over all the controllable DERs and a cost associated with the total network load, subject to voltage regulation constraints. The proposed implementation is scalable to large distribution feeders comprising networked AGs.  It largely reduces the computational burden compared to the centrally coordinated primal-dual algorithm by utilizing the information structure of the AGs. The performance of our design is analytically characterized and numerically corroborated. The significant improvement in convergence speed shows the great potential of the proposed method for grid optimization and control in real time.

\section*{Acknowledgments}
This work was authored in part by the National Renewable Energy Laboratory, operated by Alliance for Sustainable Energy, LLC, for the U.S. Department of Energy (DOE) under Contract No. DE-EE-0007998. Funding provided by U.S. Department of Energy. The views expressed in the article do not necessarily represent the views of the DOE or the U.S. Government. The U.S. Government retains and the publisher, by accepting the article for publication, acknowledges that the U.S. Government retains a nonexclusive, paid-up, irrevocable, worldwide license to publish or reproduce the published form of this work, or allow others to do so, for U.S. Government purposes.
%This work was supported by U.S. Department of Energy through the ENERGISE program and the Advanced Grid Modeling program.
%{\color{red}Here I (Wenbo) think we one or two sentences about some application potential or implication of the simulations and algorithm, otherwise the conclusion looks like an abstract. Also, Fig4 fonts are too small.}

%%%%%%%%%%%%%%%%%%%%%%%%%%%%%%%%%%%%%%%%%%%%%
\bibliographystyle{IEEEtran}
\bibliography{biblio.bib}
%%%%%%%%%%%%%%%%%%%%%%%%%%%%%%%%%%%%%%%%%%%%%
%\newpage
\appendix
\subsection{Proof of Lemma~\ref{lem:mono}}
\begin{proof}
Let $f(\bm{y})=\sum_{i\in\cN} C_i(p_i,q_i)+C_0(P_0(\bm{p},\bm{q}))$, and $\bm{\mu}^{\top}\bm{g}(\bm{y})=\underline{\bm{\mu}}^{\top}(\underline{\bm{v}}-\bm{v}(\bm{p},\bm{q}))+\overline{\bm{\mu}}^{\top}(\bm{v}(\bm{p},\bm{q})-\overline{\bm{v}})$ for simplicity. 
%Based on assumptions, $f(\bm{x})$ is strongly convex in $\bm{y}$.
Then $T(\bm{z})$ can be equivalently decomposed into the following operators:
\begin{eqnarray}
T(\bm{z})&\hspace{-3mm}=&\hspace{-3mm}\begin{bmatrix}\nabla_{\bm{y}} f(\bm{y}) \\ \nabla_{\bm{\mu}}\frac{\phi}{2}\|\bm{\mu}\|_2^2\end{bmatrix}+\begin{bmatrix}\nabla_{\bm{y}} \bm{\mu}^{\top}\bm{g}(\bm{y}) \\ -\nabla_{\bm{\mu}}\bm{\mu}^{\top}\bm{g}(\bm{y})\end{bmatrix}\nonumber\\
&\hspace{-3mm}=&\hspace{-3mm}\begin{bmatrix}\nabla_{\bm{y}} f(\bm{y}) \\ \nabla_{\bm{\mu}}\frac{\phi}{2}\|\bm{\mu}\|_2^2\end{bmatrix}+\begin{bmatrix} 0 & 0 & -R^{\top} & R^{\top}\\ 0 & 0 & -X^{\top} & X^{\top}\\ R & X & 0 & 0 \\ -R & -X & 0 & 0 \end{bmatrix}\begin{bmatrix}\bm{p}\\\bm{q}\\\underline{\bm{\mu}}\\\overline{\bm{\mu}}\end{bmatrix}\nonumber\\
&&+\ \text{Constant}.
\end{eqnarray}

One can verify that the first operator $\begin{bmatrix}\nabla_{\bm{y}} f(\bm{y}) \\ \nabla_{\bm{\mu}}\frac{\phi}{2}\|\bm{\mu}\|_2^2\end{bmatrix}$ is strongly monotone since $f(\bm{y})$ and $\frac{\phi}{2}\|\bm{\mu}\|_2^2$ are strongly convex in $\bm{y}$ and $\bm{\mu}$, respectively. The second (linear) operator is monotone since
\begin{eqnarray}
\begin{bmatrix*}[c] 0  &\hspace{-2mm} 0 &\hspace{-2mm} -R^{\top} &\hspace{-2mm} R^{\top}\\ 0 &\hspace{-2mm} 0 & \hspace{-2mm} -X^{\top} &\hspace{-2mm} X^{\top}\\ R &\hspace{-2mm} X &\hspace{-2mm} 0 &\hspace{-2mm} 0 \\ -R &\hspace{-2mm} -X & \hspace{-2mm}0 &\hspace{-2mm} 0 \end{bmatrix*}+\begin{bmatrix*}[c] 0 & \hspace{-2mm}0 & \hspace{-2mm}-R^{\top} &\hspace{-2mm} R^{\top}\\ 0 &\hspace{-2mm} 0 & \hspace{-2mm}-X^{\top} & \hspace{-2mm}X^{\top}\\ R & \hspace{-2mm}X & \hspace{-2mm}0 & \hspace{-2mm}0 \\ -R & \hspace{-2mm}-X &\hspace{-2mm} 0 &\hspace{-2mm} 0 \end{bmatrix*}^{\top}\hspace{-1mm} \succeq 0. \nonumber
\end{eqnarray}

Therefore, $T(\bm{z})$ is a strongly monotone operator as the result of combining a strongly monotone operator and a monotone operator.
\end{proof}

\subsection{Proof of Theorem~\ref{the:converge}}
\begin{proof}
Let $\Delta<1$ be some positive constant. We have
\begin{eqnarray}
\hspace{-3mm}&&\|\bm{z}(t+1)-\bm{z}^*\|_2^2\nonumber\\
\hspace{-3mm}&\leq& \|\bm{z}(t)-\epsilon T(\bm{z}(t))-\bm{z}^*+\epsilon T(\bm{z}^*)\|_2^2\nonumber\\
\hspace{-3mm}&=& \|\bm{z}(t)-\bm{z}^*\|_2^2+\|\epsilon T(\bm{z}(t))-\epsilon T(\bm{z}^*)\|_2^2\nonumber\\
\hspace{-3mm}&&\hspace{3mm}-2\epsilon (\bm{z}(t)-\bm{z}^*)^{\top}(T(\bm{z}(t))- T(\bm{z}^*))\nonumber\\
\hspace{-3mm}&\leq&(1+\epsilon^2L^2-2\epsilon M)\|\bm{z}(t)-\bm{z}^*\|_2^2\nonumber
%+\epsilon^2L^2\|\bm{z}(t)-\bm{z}^*\|_2^2-2\epsilon M\|\bm{z}(t)-\bm{z}^*\|_2^2,\nonumber
%\hspace{-3mm}&<& \Delta \|\bm{z}(t)-\bm{z}^*\|_2^2\nonumber,
\end{eqnarray}
where the first inequality comes from non-expansiveness of projection operator, and the second from \eqref{eq:strongmono} and \eqref{eq:lipschitz}. Notice that  \eqref{eq:strongmono} and \eqref{eq:lipschitz} together also guarantee that $M\leq L$ so that $1+\epsilon^2L^2-2\epsilon M$ is always nonnegative. Then based on the condition for Theorem~\ref{the:converge} that $0<\epsilon\leq \overline{\epsilon}<2M/L^2$ one has $\|\bm{z}(t+1)-\bm{z}^*\|_2^2\leq \Delta \|\bm{z}(t)-\bm{z}^*\|_2^2$ for some constant $0<\Delta <1$.

Therefore, \eqref{eq:mapk} converges to the unique saddle point of \eqref{eq:langr} exponentially fast.
\end{proof}

\end{document}

%% file: introduction_v1.tex
\section{Introduction}

The increasing penetration of distributed energy resources (DERs)---such as rooftop photovoltaic, electric vehicles, battery energy storage systems, thermostatically controlled loads, and other controllable loads---has not only provided enormous potential control flexibility that we can explore, but also imposed challenging tasks of optimally coordinating a large number of networked endpoints to satisfy system-wide objectives and constraints such as demand response and voltage regulation. Distributed algorithms are developed to facilitate the scalable control of large networks of dispatchable DERs by distributing the computational burden either coordinated by a (logically) central controller, e.g.,  \cite{ bolognani2013distributed,dall2018optimala,zhou2017incentive, zhou2017discrete}, or among neighboring agents without a central controller, e.g., \cite{peng2018distributed,magnusson2017voltage, vsulc2014optimal,bazrafshan2017decentralized, wu2018smart}. In both cases however, the computational loads increase as the system gets larger, making it more difficult to realize fast real-time control.

One way to cope with the computational burden of large systems is to consider the system as a network of autonomous grids (AGs).
AGs rely on cellular building blocks that can self-optimize when isolated from neighboring grids and can participate in optimal operations when interconnected to a larger grid \cite{kroposki2018autonomous}. 
Such structure enables us to divide the potentially immense computation to smaller problems for AGs.
%We are able to divide a large distribution system into networked AGs, each functioning as an AG while being coordinated for system-wide control goals.

%Through cellular autonomous operation, AGs are not only able to protect themselves from potential blackouts or cyber-attacks of outer grids, but can also help stabilize the whole utility networks through mitigating intermittence of renewable energy resources from within cells \cite{kroposki2018autonomous}. 

Extensive studies have been done on optimization and  control for islanded AGs, e.g., on frequency and voltage regulation within microgrids \cite{dorfler2016breaking,simpson2017voltage} , and  on technologies for AGs to follow dispatch power set points from the bulk system operator while respecting operational constraints inside AGs \cite{dall2018optimal,singhal2017framework}. %Market mechanism for voltage regulation has been designed in \cite{zhou2017incentive}.
On the other hand, interactions among AGs in grid-connected mode are crucial for the optimality and stability of a larger network of AGs. References \cite{wang2015coordinated,wang2016networked,fathi2013adaptive} consider centralized optimal control to balance energy among AGs by simplifying each AG as a node, but they do not model power flow  within individual AGs.
%but lacks analytical performance characterization. Moreover, the physical dynamics within individual AGs are not considered.
References \cite{utkarsh2018distributed,shi2015distributed,wang2016incentivizing, gregoratti2015distributed, zhang2014randomized, wang2014game} design distributed algorithms to balance power transmission among AGs while respecting operational constraints within AGs without considering power flow dynamics among AGs.
%design distributed algorithms to balance power transmission among AGs while maintaining operational constraints within AGs, but do not consider
%physical dynamics among AGs.
%\cite{chiu2015multiobjective} models multi-microgrids system as a multiobjective approach to achieve the Pareto optimal, and centralized algorithm is applied.
Studies have also designed incentive-based  distributed algorithms to optimize energy transactions among AGs without modeling the power flow among AGs  \cite{wang2016incentivizing, gregoratti2015distributed, zhang2014randomized, wang2014game}.

It is crucial to model the power flow both within and among AGs to satisfy network-wide objectives and operational constraints, e.g., voltage regulation; however, very limited literature has done this. Recently, \cite{zhang2018dynamic} applies a game-theoretic approach to manage a partitioned distribution network based on noncooperative Nash game, but the uniqueness of the equilibrium, the convergence, and the global performance are all difficult to characterize.
%cite{scutari2010convex}.

This work considers a potentially large distribution network controlled cooperatively by a number of networked AGs. We model the distribution network using the DistFlow model \cite{baran1989optimala, baran1989optimalb}, which captures power flows within and among AGs. A regional coordinator (RC) communicates with all the dispatchable nodes within each AG, and a central coordinator (CC) communicates with all the RCs. Each RC knows only the topology and line parameters of the AG that it controls, and the CC knows only the topology and line parameters  of the reduced network, which treats each AG as a node and connects all the AGs. Given such information availability, we explore the topological structure of the linearized DistFlow (LinDistFlow) model to derive a hierarchical, distributed implementation of the primal-dual gradient algorithm that solves an optimal power flow (OPF) problem.  
%, and based on the primal-dual gradient algorithm we design a hierarchical distributed algorithm that solves an optimal power flow (OPF) problem that
The OPF problem minimizes the total cost of all the controllable DERs and a cost associated with the total network load subject to voltage regulation constraints. The proposed implementation significantly reduces the computation burden compared to the centrally coordinated  implementation of the primal-dual algorithm, which requires a central coordinator for the whole network. The performance of the proposed implementation is verified through numerical simulation of a 4,521-node test feeder. Simulation results show that an improvement of more than tenfold in the speed of convergence can be achieved by the hierarchical distributed method compared to the centrally coordinated implementation. This significant improvement in convergence speed makes real-time grid optimization and control possible. Meanwhile, to our best knowledge, the size of the network in our simulation is the largest in optimization-based control in power system.

%all nodes and the cost of lineloss with voltage regulation constraints. Performance is analytically characterized and numerical verified through a 11,000 nodes test feeder with more than 1,000 control nodes, showing a 33-fold speed improvement compared with the original primal-dual gradient algorithm on solving the same OPF problem. The significant convergence speed improvement makes real-time grids optimization and control possible. Meanwhile to our best knowledge, the size of the network in our simulation is the largest in optimization-based control in power system.

The rest of this paper is organized as follows. Section~\ref{sec:model} models the distribution system, formulates the OPF problem, and introduces the primal-dual gradient algorithm for solving the OPF problem. Section~\ref{sec:hier} proposes a hierarchical distributed implementation of the primal-dual gradient algorithm.
% to equivalently solve the same OPF problem. 
%Extensions of the design to multi-level control and distributed state estimation are briefly introduced in Section~\ref{sec:ext}. 
Section~\ref{sec:numerical} provides numerical results, and Section~\ref{sec:conclusion} concludes this paper.